\documentclass[10pt,reqno]{amsart}
\usepackage{amsmath} 
\usepackage{amssymb}
\usepackage{dsfont}
\usepackage[dvips,draft,final]{graphics}
\usepackage[T1]{fontenc}
\usepackage{fancyhdr}
\usepackage{url}

\usepackage{color}
\usepackage{graphicx}
\newtheorem{thm}{Theorem}[section]
\newtheorem{lem}[thm]{Lemma}
\newtheorem{cor}[thm]{Corollary}
\newtheorem{prop}[thm]{Proposition}
\newtheorem{defn}[thm]{Definition}
\newtheorem{rmk}{Remark}

\numberwithin{equation}{section}

\newcommand{\R}{\mathbb{R}}

\allowdisplaybreaks
\def\epsilon{\varepsilon}
\def\phi {\varphi}

\providecommand{\norm}[1]{\left\lVert#1\right\rVert}

\renewcommand{\leq}{\leqslant}
\renewcommand{\geq}{\geqslant}

\providecommand{\norm}[1]{\left\lVert#1\right\rVert}

\begin{document}

\title[Reconstruction and stable recovery of source terms ]{Reconstruction and stable recovery of source terms and coefficients appearing in diffusion equations}

\author[Yavar Kian]{Yavar Kian}
\address{Aix Marseille Univ, Universit\'e de Toulon, CNRS, CPT, Marseille, France}
\email{yavar.kian@univ-amu.fr}

\author[Masahiro Yamamoto]{Masahiro Yamamoto}
\address{Graduate School of Mathematical Sciences,
The University of Tokyo, 3-8-9 Komaba Meguro, Tokyo 153-8914, Japan\\
Honorary Member of Academy of Romanian Scientists, Splaiul Independentei 
Street, no 54, 050094 Bucharest Romania \\
Peoples' Friendship University of Russia (RUDN University)
6 Miklukho-Maklaya St, Moscow, 117198, Russian Federation}
\email{myama@ms.u-tokyo.ac.jp}

\baselineskip 18pt

\begin{abstract}  We consider the inverse source problem of determining 
a source term depending on both time and space variables for fractional and 
classical diffusion equations in a cylindrical domain from boundary 
measurements. With suitable boundary conditions we prove that some class of 
source terms which are independent of one space direction, can be 
reconstructed from boundary measurements. Actually, we prove that this 
inverse problem is well-posed. We establish also some results of Lipschitz 
stability for the recovery of source terms which we apply to the stable 
recovery of time-dependent coefficients.\\



\end{abstract}
\maketitle
\section{Introduction}

\subsection{Statement}
Let $d \geq 2$, $\Omega=\omega\times (-\ell,\ell)$ and
$\tilde{\Omega} = \omega \times (0,\ell)$, and $\omega \subset 
\R^{d-1}$ be a bounded domain with $\mathcal C^2$ boundary. 
We set $Q = (0,T) \times \Omega$ and $\tilde{Q} = (0,T) \times \tilde{\Omega}$.
Let $\nu = \nu(x)$ be the outward unit normal vector to $\partial\Omega$ 
or $\partial\tilde{\Omega}$ at $x$. 
In what follows, we define $\mathcal A$ by the differential operator 
\[\mathcal Au(x)=-\sum_{i,j=1}^d\partial_{x_i}\left(a_{ij}(x)\partial_{x_j}u\right),\quad x\in\Omega,\]
where $a_{ij}=a_{ji}\in \mathcal C^2(\overline{\omega}\times[-\ell,\ell])$,
$1 \le i,j \le d$, satisfy
\[\sum_{i,j=1}^da_{ij}(x)\xi_i\xi_j\geq c|\xi|^2,\quad x\in\overline{\Omega},\ \xi=(\xi_1,\ldots,\xi_d)\in\R^d,\]
\begin{equation}\label{aa}a_{jd}(x',\pm\ell)=0,\quad a_{dd}(x',\pm\ell)>0,\quad  x'\in\overline{\omega},\ j=1,\ldots,d-1.\end{equation}
For $\alpha\in(0,2)\setminus\{1\}$, we denote by $\partial_t^\alpha$  the Caputo fractional derivative with respect to $t$ given by
\[
\partial_t^\alpha u(t,x):=\frac{1}{\Gamma([\alpha]+1-\alpha)}\int_0^t(t-s)^{[\alpha]-\alpha}\partial_s^{[\alpha]+1} u(s,x) d s,\ (t,x) \in Q,
\]
and by  $\partial_t^1$ the usual derivative in $t$.
We set 
$$
m_{\alpha} = 
\left\{ \begin{array}{rl}
0, \quad &0 < \alpha \leq 1, \\
1, \quad &1<\alpha < 2,
\end{array}\right.
$$
$T\in(0,+\infty)$,  $\Sigma=(0,T)\times\partial\Omega$, 
$0<\alpha<2$ and we consider the following problem 
\begin{equation}\label{eq11}
\left\{\begin{aligned}\partial_t^\alpha u+\mathcal A u=F(t,x),\quad 
&(t,x',x_d)=(t,x)\in 
Q,\\  
\partial_t^ku(0,x)=0,\quad &x\in\Omega,\ k=0,\ldots,m_\alpha.
\end{aligned}\right.
\end{equation}
We associate with this problem the following  boundary conditions
\begin{equation}\label{b1}
\partial_{x_d}^mu(t,x',\pm \ell)=0,\quad(t,x')\in(0,T)\times \omega,\end{equation}
\begin{equation}\label{b2}
\partial_\nu^n u(t,x)=0,\quad (t,x)\in(0,T)\times \partial\omega
\times(-\ell,\ell), 
\end{equation}
with $(m,n)\in\{0,1\}^2$.
In the same way, we consider the problem
\begin{equation}\label{eq12}
\left\{\begin{aligned}\partial_t^\alpha u+\mathcal A u+q(t,x')u=F(t,x),\quad &(t,x',x_d)=(t,x)\in 
\tilde{Q},\\  
\partial_t^ku(0,x)=0,\quad &x\in\tilde{\Omega},\ k=0,\ldots, m_\alpha,
\end{aligned}\right.
\end{equation}
\begin{equation}\label{b3}\partial_{x_d}u(t,x', 0)=u(t,x', \ell)=0,\quad(t,x')\in(0,T)\times \omega,\end{equation}
\begin{equation}\label{b4} u(t,x)=0,\quad (t,x)\in(0,T)\times 
\partial\omega\times(0,\ell).\end{equation}
We refer to Section 1.5 and Proposition \ref{l1} for the definition of solutions of problem \eqref{eq11}-\eqref{b3} (resp. \eqref{eq12}-\eqref{b4}) as well as the existence and uniqueness of solutions.   This verifies that 
an initial-boundary value problem \eqref{eq12}-\eqref{b4} well defines a map
from $F$ to the solution $u(t,x)$. In the present paper, we treat the inverse problem of determining the source 
term $F$ or the coefficient $q$ from measurements of the solution of \eqref{eq11}-\eqref{b2} or \eqref{eq12}-\eqref{b4} on a subboundary
of the cylindrical domain $Q$ or $\tilde{Q}$.

\subsection{Obstruction against the uniqueness}
We recall that there is an obstruction against 
the recovery of general source terms 
$F$ from any type of measurements of the solution of \eqref{eq11}-\eqref{b2} (resp. \eqref{eq12}-\eqref{b4}) on $(0,T)\times\partial\Omega$ (resp. $(0,T)\times\partial\tilde{\Omega}$). Indeed, choose $\chi \ne 0, \in\mathcal C^\infty_0(Q)$
and consider $F:=\partial_t^\alpha \chi+\mathcal A\chi$. From the uniqueness of the weak solution  of problem \eqref{eq11}-\eqref{b2} (see Section 1.4 for more details and see Definition \ref{d1} below for the definition of weak solutions), we knows that $u=\chi$ and, since $\chi\neq0$, we deduce that $F\neq0$. However, we have
$$\partial_\nu^ku(t,x)=\partial_\nu^k\chi(t,x)=0,\quad (t,x)\in(0,T)\times\partial\Omega,\quad k=0,1,\ldots,$$
with $\nu$ the outward unit normal vector of $\partial\Omega$.  

Facing this obstruction against the uniqueness, 
we will consider source terms of the form 
\begin{equation}\label{source}F(t,x',x_d):=f(t,x')R(t,x',x_d),\quad (t,x',x_d)\in(0,T)\times\omega\times (-\ell,\ell)\end{equation}
and, assuming that $R$ is known, we will consider the problem of determining 
$f$.

\subsection{Motivations}
We recall that the problem \eqref{eq11}-\eqref{b2} 
(resp. \eqref{eq12}-\eqref{b4}) is associated with different models of 
diffusion. The non-integer value of the power  $\alpha$, is frequently used 
for describing anomalous diffusion derived from continuous-time random walk 
models (see e.g., \cite{MK}). In this context the recovery of the source term 
$f$ can be seen as the recovery of a time evolving source of diffusion. 
For instance, in the case $\alpha=1$ and $d=3$, our problem can be associated 
with the recovery of a  source moving in the subset $\omega$ of the plan 
$\R^2$ from a single measurement of the heat flux at the boundary. Such a problem can be associated with the determination of different properties of 
materials such as metal (see e.g., \cite{Ki} for the heat equation). 
We refer to \cite{HKLZ} for applications of recovery of source terms of 
the form \eqref{source} to the recovery of moving sources in the 
electrodynamics.
For non-integer value of the power $\alpha$, in the spirit of \cite{NSY} 
(see also \cite{JR}), our inverse problem can  be seen as the recovery of 
a moving source  of diffusion of a contaminant under the ground. 
As unknown sources, we assume the form of (\ref{source}),
which can be interpreted for example that an unknown source $f(t,x')$ depends 
only on the depth variable and $t$ in the case of $d=2$, which corresponds 
to a layer structure, and on the planar locations 
$(x_1,x_2)$ and $t$ but not on the depth in the case of $d=3$, which may be 
a good approximation if $\Omega$ is a very thin domain 
in the direction of $x_3$.
 
\subsection{Known results}
Inverse source problems have received a lot of  attention these last decades 
among the mathematical community (see \cite{I} for an overview). 
For diffusion equations corresponding to the case $\alpha=1$ with 
time independent source terms, several authors investigated the conditional 
stability (e.g. 
\cite{choY,yam,yam1}).  Following the Bukhgeim-Klibanov approach introduced 
in \cite{BK}, \cite{IY} established Lipschitz stable recovery of the source 
from one Neumann boundary measurement. 
In \cite{cho}, the authors derived a stability estimate for this problem 
from a single Neumann observation of the solution on an arbitrary portion of 
the boundary.  For fractional diffusion equations corresponding to 
the case $\alpha\neq1$, \cite{TU1,TU2} proved the recovery of a time independent source term appearing in some class of  one dimensional time-space fractional diffusion equations while \cite{JLLY} treated this problem in the multi-dimensional case ($d\geq2$) for time fractional diffusion equations.
Despite  the physical backgrounds related to 
various anomalous diffusion phenomena stated above,
to our best knowledge, there is no result in the mathematical literature 
dealing with  the recovery of source terms, depending on both time and space 
variables, of the form \eqref{source}, for fractional diffusion equations.
Despite  the physical backgrounds related to 
various anomalous diffusion phenomena stated above,
to our best knowledge, there is no result in the mathematical literature 
dealing with  the recovery of source terms, depending on both time and space 
variables, of the form \eqref{source}, for fractional diffusion equations.
We mention also the work of \cite{ JR0, JR, KOSY, KSY,LKS, RZ} where some 
inverse coefficient problems and some related results have been considered.

The above mentioned results are all concerned with the determination of
time independent source terms $f(x)$. 
Several authors considered also the recovery of time-dependent source terms. 
In \cite{FK,SY} the authors proved the stable recovery of a source term 
$f(t)$ depending only on the time variable from measurements of solutions
at one spatial point over time interval. 
As long as the classical partial differential equations with natural 
number $\alpha$ are concerned, 
some papers have also been devoted to the unique existence
and the stability in finding source terms of the form 
\eqref{source} (see e.g. \cite{Bez, HK,I}). In particular, 
in \cite[Section 6.3]{I} the author proved the reconstruction and the unique 
recovery of source terms of the form 
\eqref{source}  appearing in a parabolic equation on the half space. 

Our main results stated below, seem to be the first achievements for the 
inverse source problem of determining $f(t,x')$ in \eqref{source} for 
fractional partial differential equations.

\subsection{Preliminary properties}

In the present paper, following \cite{KY,SY}, we consider solutions of problem \eqref{eq11}-\eqref{b3} (resp. \eqref{eq12}-\eqref{b4} with $q=0$) in the following weak sense.

\begin{defn}\label{d1} 
Let $F\in L^2(Q)$ (resp. $F\in L^2((0,T)\times\tilde{\Omega})$). We say that problem \eqref{eq11}-\eqref{b2} $($resp.  \eqref{eq12}-\eqref{b4} with $q=0$$)$ admits a  weak solution $u$ if there exists $v\in L^2_{\textrm{loc}}(\R^+;L^2(\Omega))$ such that:\\
1) $v_{\vert Q}=u$ and $\inf\{\epsilon>0:\ e^{-\epsilon t}v\in L^1(\R^+; L^2(\Omega))\}=0$,\\
2) for all $p>0$ the Laplace transform $V(p)=\int_0^{+\infty}e^{-pt}v(t,.)dt$
with respect to $t$ of $v$, satisfies \eqref{b1}-\eqref{b2} (resp.  \eqref{b3}-\eqref{b4}) and solves
\[
(\mathcal A +p^\alpha)V(p)=\hat{F}(p) \quad \textrm{in }\Omega\ (\textrm{resp. in }\tilde{\Omega}),
\]
where $\hat{F}(p)=\mathcal L[F(t,.)\mathds{1}_{(0,T)}(t)](p)=\int_0^{T}e^{-pt}F(t,.)dt$ and $\mathds{1}_{(0,T)}$ is the characteristic function 
of $(0,T)$. 
\end{defn}
By the results in \cite{LM1,SY}, we can prove that for $F\in L^2(Q)$  problem \eqref{eq11}-\eqref{b2} (resp. \eqref{eq12}-\eqref{b4} with $q=0$) admits a unique solution $u\in L^2(0,T;H^1(\Omega))$ satisfying $\partial_t^\alpha u,\mathcal A u\in L^2(Q)$. For sake of completeness we recall this result in the Appendix.

In \eqref{eq12}, we assume that  $a_{ij}\in\mathcal C^3(\overline{\Omega})$, $i,j\in\{1,\ldots,d\}$, satisfy the following condition
\begin{equation}\label{co2} 
a_{ij}(x',-x_d)=a_{ij}(x',x_d),\quad (x',x_d)
\in\overline{\tilde{\Omega}},\ i,j\in\{1,\ldots,d\}.
\end{equation}
Then we consider the problem

\begin{equation}\label{eq21}
\left\{\begin{aligned}\partial_t^\alpha u+\mathcal A u=F(t,x),\quad &(t,x)\in 
(0,T)\times \tilde{\Omega},\\  
u(t,x)=0,\quad &(t,x)\in(0,T)\times [\partial\omega\times(0,\ell)\cup \omega\times\{\ell\}],\\ 
-\partial_{x_d}u(t,x',0)=0\quad &(t,x')\in(0,T)\times \omega,\\
\partial_t^ku(0,x)=0,\quad &x\in\Omega,\ k=0, m_\alpha.\end{aligned}\right.
\end{equation}

Here we consider weak solutions in the sense of Definition \ref{d1} and, following \cite{KY,SY}, we can prove that there exists an operator valued function $J(t)\in\mathcal B(L^2(\tilde{\Omega}))$ such that
the solution of \eqref{eq21} takes the form
$$u(t,\cdot)=\int_0^tJ(t-s)F(s,\cdot)ds.$$
Using this definition, for $q\in L^\infty((0,T)\times\omega)$ we can define the solution $u$ of \eqref{eq12}-\eqref{b4} in the mild sense as a solution of the integral equation
$$u(t,\cdot)=-\int_0^tJ(t-s)qu(s,\cdot)ds+\int_0^tJ(t-s)F(s,\cdot)ds.$$

\subsection{Main results}
From now on, we assume that $F$ takes the form \eqref{source}. For our first result we need an assumption on $\omega$ and $\mathcal A$ that guarantees the elliptic regularity of the operator $\mathcal A$.  Indeed, due to the fact that the 
domain $\Omega$ is only Lipschitz, some extra assumptions will be required for 
guaranteeing the elliptic regularity of $\mathcal A$ with the boundary conditions \eqref{b1}-\eqref{b2}. For this purpose, for $m,n=0,1$, 
we introduce the condition (H$mn$) (in (H$mn$), $m,n$ denote the numbers $m$ and $n$) corresponding to the requirement that for all $v\in H^1(\Omega)$ satisfying $\mathcal Av\in L^2(\Omega)$ and 
\eqref{b1}-\eqref{b2} with these values of $m,n$, 
we have $v\in H^2(\Omega)$ and there exists $C>0$ depending only on 
$\mathcal A$, $m$, $n$ and $\Omega$ such that
$$\norm{ v}_{H^2(\Omega)}\leq C(\norm{\mathcal A v}_{L^2(\Omega))}+\norm{v}_{L^2(\Omega))}).$$

Note that conditions (H00) and (H11) will be fulfilled if, for instance, 
we assume that $\omega$ is convex. Indeed, in that case $\Omega$ will also
be convex and, in virtue of \cite[Theorem 3.2.1.2]{Gr} and \cite[Theorem 3.2.1.3]{Gr}, (H00) and (H11) will be fulfilled. In the same way, assuming that 
\begin{equation}\label{eli}a_{id}=0,\quad \partial_{x_j}a_{dd}=0,\quad \partial_{x_d}a_{ij}=0,\quad i,j\in\{1,\ldots,d-1\},\end{equation}
we deduce from a separation of variable argument similar to \cite[Lemma 2.4]{GK} that, for all $m,n=0,1$, (H$mn$) is fulfilled.

Using the conditions (H$mn$), we obtain the following.

\begin{thm}\label{t3} 
Let  \emph{(H00)}, \emph{(H10)} be fulfilled and assume that $R,\partial_{x_d}R\in L^\infty(Q)$ and there exists a constant $c>0$ such that 
\begin{equation}\label{t1a}|R(t,x',\ell)|\geq c,\quad (t,x')\in(0,T)\times\omega.\end{equation}
Assume also that  the condition
\begin{equation}\label{t3a}\partial_{x_d}a_{ij}=0,\quad i,j=1,\ldots,d-1,
\end{equation}
is fulfilled. Then, for $f\in L^2((0,T)\times\omega)$, the solution $u$ 
of \eqref{eq11}-\eqref{b2} with $m=1$ and $n=0$ satisfies 
$u\in H^3(-\ell,\ell;L^2((0,T)\times\omega))$, $\partial_t^\alpha u, \mathcal A u\in H^1(-\ell,\ell;L^2((0,T)\times\omega))$. Therefore, we can define 
\begin{equation}\label{t3b}h:=(t,x')\mapsto\frac{[\partial_t^\alpha u+(\mathcal A+a_{dd}\partial_{x_d}^2)u](t,x',\ell)}{R(t,x',\ell)}\in L^2((0,T)\times\omega).\end{equation}
Moreover, we
can define an  operator  $\mathcal H\in \mathcal B(L^2((0,T)\times\omega))$,  such that   $f$ solves the  equation 
\begin{equation}\label{t3c}f=h+\mathcal H f,\end{equation}
which is well-posed.
Finally, for every $(h,f)\in L^2((0,T)\times\omega)\times L^2((0,T)\times\omega)$ satisfying \eqref{t3c} the solution $u$ of \eqref{eq11}-\eqref{b2} satisfies \eqref{t3b}.
In the same way, assuming that  \emph{(H01)}  and \emph{(H11)} are fulfilled, the same results hold true for the problem \eqref{eq11}-\eqref{b2}
with $m=1$ and $n=1$.
\end{thm}

\begin{rmk} Note that the data $h$ depends only on $\mathcal A$, $R(t,x',\ell)$ and $u(t,x',\ell)$, $(t,x')\in(0,T)\times\omega$. Indeed, thanks to the condition $\partial_{x_d}u(t,x',\ell)=0$,  the expression $(\mathcal A+a_{dd}\partial_{x_d}^2)u(t,x',\ell)$ depends only on $\mathcal A$ and $u(t,x',\ell)$, $(t,x')\in(0,T)\times\omega$. Therefore, assuming that $\mathcal A$ and $R$ are known, the result of Theorem \ref{t3} can be seen has a result of reconstruction of $f$ from the data $u(t,x',\ell)$, $(t,x')\in(0,T)\times\omega$.\end{rmk}

For problem \eqref{eq12}-\eqref{b4}, we consider first the following condition:\\
\ \\
$(\tilde{H})$ For all $v\in H^{\max(1,s)}(\tilde{\Omega})$ satisfying $\mathcal Av\in H^s(\tilde{\Omega})$, $s\in[0,2]$, and \eqref{b3}-\eqref{b4}, we have $v\in H^{2+s}(\Omega)$ and there exists $C>0$ depending only on $\mathcal A$, $s$ and $\Omega$ such that
$$\norm{ v}_{H^{2+s}(\Omega))}\leq C(\norm{\mathcal A v}_{H^s(\Omega))}+\norm{v}_{H^s(\Omega))}).$$
Assuming that $\omega$ is of class $\mathcal C^4$ and using  a separation of variable argument similar to \cite[Lemma 2.4]{GK}, one can check that \eqref{eli} implies $(\tilde{H})$.

Using $(\tilde{H})$ we obtain the following well-posedness result.

\begin{prop}\label{l1} 
Assume that $(\tilde{H})$ is fulfilled. Let $\gamma\in(0,1)$ be such that for $\alpha\in(0,1]$, $\gamma\in(1/2,1)$ and  for $\alpha\in(1,2)$, $\gamma\in(1/2,1/\alpha)$.  Fix $p\in(1,+\infty)$ such that $\frac{1}{p}<\min(1-\alpha\gamma,\alpha(1-\gamma))$ and let $F\in W^{1,p}(0,T; L^2(\tilde{\Omega}))\cap 
\mathcal C([0,T];H^{2\gamma}(\tilde{\Omega}))$ satisfy $F_{|t=0}=0$. 
Let $q\in \mathcal C^1([0,T]; L^\infty(\omega))\cap \mathcal C([0,T];W^{2,\infty}(\omega))$.  Then problem \eqref{eq12}-\eqref{b4} admits a unique weak  solution $u\in \mathcal C([0,T];H^{2(1+\gamma)}( \tilde{\Omega}))\cap \mathcal C^1([0,T];H^\gamma(\tilde{\Omega}))$.
\end{prop}

Applying this well-posedness result, we can state our second main result as 
follows.

\begin{thm}\label{t1} 
Assume that $(\tilde{H})$,  \emph{(H00)}, \eqref{co2} and \eqref{t3a} are fulfilled, $\alpha\in(0,1]$, $\omega$ is  $\mathcal C^4$ and $\gamma\in(3/4,1)$. 
Fix $p\in(1,+\infty)$ such that $\frac{1}{p}<\min(1-\alpha\gamma,\alpha(1-\gamma))$.  Let $q\in \mathcal C^1([0,T]; L^\infty(\omega))\cap \mathcal C([0,T];W^{2,\infty}(\omega))$, $f\in W^{1,p}(0,T; L^\infty(\omega))\cap \mathcal C([0,T];W^{2,\infty}(\omega))$ and 
$R\in W^{1,\infty}(0,T; L^2(\tilde{\Omega}))\cap 
\mathcal C([0,T]; H^{2\gamma}(\tilde{\Omega}))$ satisfy 
$R,\partial_{x_d}R\in L^\infty((0,T)\times\tilde{\Omega})$ and \eqref{t1a}. 
Assume also that
$$f(0,x)=0,\quad x\in\Omega$$
and let $u\in \mathcal C([0,T];H^{2(1+\gamma)}( \tilde{\Omega}))\cap \mathcal C^1([0,T];H^\gamma(\tilde{\Omega}))$ be the solution of \eqref{eq12}-\eqref{b4}.
Then, for $\delta\in \left(0,2\gamma-\frac{3}{4}\right)$, there exists a constant $C$  depending on  $q$,  $R$, $\tilde{\Omega}$, $T$, $\alpha$, $\mathcal A$, $\delta$, such that
 \begin{equation}\label{t1b}\norm{f}_{L^\infty(0,T;L^2(\omega))}\leq C(\norm{\partial_{x_d}u(\cdot,\cdot,\ell)}_{L^\infty(0,T;H^{\frac{3}{2}}(\omega))}+\norm{\partial_{x_d}u(\cdot,\cdot,\ell)}_{W^{1,\infty}(0,T;H^\delta(\omega))}).\end{equation}
\end{thm}
Applying this result, we can also prove the stable recovery of 
the coefficient $q$ appearing in the problem

\begin{equation}\label{eq14}
\left\{\begin{aligned}\partial_t^\alpha v+\mathcal A v+q(t,x')v=0,\quad &(t,x)\in 
(0,T)\times \tilde{\Omega},\\  
v(t,x', \ell)=h_0(t,x')\quad &(t,x')\in(0,T)\times \omega,\\
v(t,x)=h_1(t,x) ,\quad &(t,x)\in(0,T)\times \partial\omega\times(0,\ell),\\ 
\partial_{x_d}v(t,x', 0)=0\quad &(t,x')\in(0,T)\times \omega,\\
v(0,x)=w_0,\quad &x\in\tilde{\Omega}\end{aligned}\right.
\end{equation}
with $\alpha\in(0,1]$, $h_k$,  $k=0,1$, $w_0$ such that there exists $H\in \mathcal C^1([0,T];H^{2+\gamma}(\tilde{\Omega}))\cap W^{{2},{p}}(0,T; H^{2}(\tilde{\Omega}))$ satisfying
\begin{equation}\label{trace}
\left\{\begin{aligned}
H(t,x', \ell)=h_0(t,x')\quad &(t,x')\in(0,T)\times \omega,\\
H(t,x)=h_1(t,x) ,\quad &(t,x)\in(0,T)\times \partial\omega\times(0,\ell),\\ 
\partial_{x_d}H(t,x', 0)=H(t,x', 0)=0\quad &(t,x')\in(0,T)\times \omega,\\
H(0,x)=w_0(x),\quad &x\in\tilde{\Omega},\\
\partial_t^\alpha H(0,x)=-\mathcal A H(0,x)-q(0,x')H(0,x),\quad &x=(x',x_d)
\in\tilde{\Omega}.\end{aligned}\right.
\end{equation}

The result for the determination of $q$ can be stated as follows.
\begin{cor}\label{c1} 
Let the condition of Theorem \ref{t1}  be fulfilled with 
$f=0$, $d\leq3$ and let  $h_k$,  $k=0,1$, $w_0$ be given by \eqref{trace}. Assume also that there exists $c>0$ such that 
the condition
\begin{equation}\label{t5a}|h_0(t,x')|\geq c,\quad (t,x')\in(0,T)\times\omega\end{equation}
is fulfilled. Fix $q_j\in   \mathcal C^1([0,T]; L^\infty(\omega))\cap \mathcal C([0,T];W^{2,\infty}(\omega))$, $j=1,2$, such that 
$$\norm{q_j}_{W^{1,\infty}(0,T; L^\infty(\omega))}+\norm{q_j}_{L^\infty(0,T;W^{2,\infty}(\omega))}\leq M,$$
$$q_1(0,x')=q_2(0,x')=q(0,x'),\quad x'\in\omega$$
and consider the solution $v_j$ of \eqref{eq14} with $q=q_j$.
 Then, for $\delta\in \left(0,2\gamma-\frac{3}{4}\right)$, we have
\begin{equation}\label{t4a}
\norm{q_1-q_2}_{L^\infty(0,T;L^2(\omega))}\leq C\left(\norm{\partial_{x_d}(v_1-v_2)_{|x_d=\ell}}_{L^\infty(0,T;H^{\frac{3}{2}}(\omega))}+\norm{\partial_{x_d}(v_1-v_2)_{|x_d=\ell}}_{W^{1,\infty}(0,T;H^\delta(\omega))}\right),
\end{equation}
where the constant $C$ depends on $\alpha$, $c$, $T$, $\Omega$, $a$, 
$\mathcal A$, $\alpha$, $M$, $h_0$, $h_1$, $w_0$, $\delta$.
\end{cor}
\subsection{Comments about our results}

To the best of our knowledge Theorem \ref{t3} and \ref{t1} are the first results  of recovery of a source term depending on both time and space variables for 
fractional diffusion equations of the form  \eqref{eq11} when $\alpha\neq1$. 
For $\alpha=1$, we refer to \cite[Section 6.3]{I} addressing this inverse 
problem with $\Omega$ corresponding to the half space and see also 
\cite{Bez}. In contrast to 
\cite{I}, we state our result on a bounded cylindrical domain and we restrict our analysis  to solutions lying in Sobolev spaces while \cite[Section 6.3]{I} is stated with H\"older continuous functions. Moreover 
our approach  admits a natural extension to fractional diffusion equations 
($\alpha\neq1$).
 
Let us remark that Theorem \ref{t3} gives a reconstruction algorithm for the recovery of the source term $f$ under consideration. It is actually stated as a well-posedness result for the pair of functions $(u,f)$ appearing in \eqref{eq11} and \eqref{source}. In contrast to Theorem \ref{t3}, Theorem  \ref{t1} provides only a stability estimate. However, Theorem  \ref{t1} can be applied to more general boundary conditions and it can also be applied to the stable recovery of a coefficient depending on both time and space variables (see Corollary \ref{c1}). 

Applying Theorem \ref{t1}, we prove in Corollary \ref{c1} the stable recovery of the coefficient of order zero $q$ provided $\partial_{x_d}q=0$. It seems that this result is the first result of stable recovery of a coefficient 
depending on both time and space variables for a fractional diffusion equation.
In \cite[Theorem 3.6]{GK}, the authors derived a similar result for the heat equation ($\alpha=1$) stated with stronger regularity conditions and measurements on both $\gamma\times (0,\ell)$ and $\omega\times\{\ell_1\}$, with $\gamma\subset\partial\omega$ and $\ell_1\in(0,\ell]$. Even, for $\alpha=1$, our result improves the one of \cite[Theorem 3.6]{GK} in terms of regularity conditions and restriction of the data.

Theorem \ref{t1} is stated only for $\alpha\in(0,1]$ but it can be extended 
to $\alpha\in(0,2)$ without any difficulty.
Indeed, one can easily extend our argumentation to the case $\alpha\in(1,2)$. In Theorem \ref{t1} we have restricted our analysis to $\alpha\in(0,1]$ in order to simplify the statement of this theorem and its proof.

The proof of Proposition \ref{l1} is based on properties of solutions of fraction diffusion and properties of Mittag-Leffler functions considered in several works like \cite{FK,KY,Ma,SY}.

\subsection{Outline}

This paper is organized as follows. In Section 2 we prove Theorem \ref{t3}. Section 3 is devoted to the proof of Theorem \ref{t1} and in Section 4 we consider 
the application of Theorem \ref{t1} stated in Corollary \ref{c1}. Finally, 
in the Appendix we recall and prove some results related to properties of solutions of fractional diffusion equations.
\section{Proof of Theorem \ref{t3}}

We start with  the first part of Theorem \ref{t3}. For this purpose, we assume that (H00), (H10), \eqref{t1a}-\eqref{t3a} are fulfilled and we will show \eqref{t3c}.  We denote by  $A$ (resp. $\tilde{A}$)  the operator $\mathcal A$ acting on $L^2(\Omega)$ with domain 
$$D(A):=\{g\in H^1(\Omega):\ \mathcal Ag\in L^2(\Omega),\  u_{|\partial\Omega}=0\}$$
$$(\textrm{resp. }D(\tilde{A}):=\{g\in H^1(\Omega):\ \mathcal Ag\in L^2(\Omega),\  u_{|\partial\omega\times(0,\ell)}=0,\partial_{x_d}u_{|x_d=\pm\ell}=0\}$$
Thanks to \eqref{aa} we know that $A$ (resp. $\tilde{A}$) are selfadjoint operators with a spectrum consisting of a non-decreasing sequence of non-negative eigenvalues $(\lambda_n)_{n\geq1}$ (resp. $(\tilde{\lambda}_n)_{n\geq1}$). Moreover, conditions (H00) and (H10) imply that $D(A)$ and $D(\tilde{A})$ embedded continuously into $H^2(\Omega)$. Let us also introduce 
an orthonormal  basis in the Hilbert space $L^2(\Omega)$ of eigenfunctions $(\phi_n)_{n\geq1}$ (resp. $(\tilde{\phi}_n)_{n\geq1}$) of $A$ (resp. $\tilde{A}$) associated to the non-decreasing sequence of  eigenvalues $(\lambda_n)_{n\geq1}$ (resp. $(\tilde{\lambda}_n)_{n\geq1}$). We consider also the operator valued function $S(t)$ (resp. $\tilde{S}(t)$) defined by
$$S(t)h=\sum_{n=1}^\infty t^{\alpha-1}E_{\alpha,\alpha}(-\lambda_nt^\alpha)\left\langle h,\phi_n\right\rangle_{L^2(\Omega)}\phi_n,\quad h\in L^2(\Omega),$$
$$\tilde{S}(t)h=\sum_{n=1}^\infty t^{\alpha-1}E_{\alpha,\alpha}(-\tilde{\lambda}_nt^\alpha)\left\langle h,\tilde{\phi}_n\right\rangle_{L^2(\Omega)}\tilde{\phi}_n,\quad h\in L^2(\Omega),$$
where $E_{\alpha,\alpha}$ corresponds to the Mittag-Leffler function given by  
$$E_{\alpha,\alpha}(z)=\sum_{n=0}^\infty \frac{z^n}{\Gamma(\alpha(n+1))},\quad z\in\mathbb C.$$
Following \cite{KY,SY} (see also Lemma \ref{l2}), one can check that problem \eqref{eq11} admits a unique weak solution $u\in  L^2(0,T;D(\tilde{A}))$ taking the form
$$u(t,\cdot)=\int_0^t\tilde{S}(t-s)F(s,\cdot)ds,\quad t\in(0,T).$$
Recall that the function $v$  given by Definition \ref{d1} takes the form
$$v(t,\cdot)=\int_0^{\min(t,T)}\tilde{S}(t-s)F(s,\cdot)ds,\quad t\in(0,+\infty).$$
Moreover, using the fact that $D(\tilde{A})$ embedded continuously into $H^2(\Omega)$, we deduce by interpolation that for $s_1\in\left(\frac{3}{4},1\right)$,  $D(\tilde{A}^{s_1})$ embedded continuously into $H^{2s_1}(\Omega)$. Therefore, applying Lemma \ref{l10} (see the Appendix), one can check that   $\inf\{\epsilon>0:\ e^{-\epsilon t}v\in L^1(\R^+; H^{2s_1}(\Omega))\}=0$.
Fixing $w=\partial_{x_d}u$ we deduce that, for all $p>0$, the Laplace transform in time  $W(p)$ of  the extension of $w$ to $\R_+\times\Omega$ given by $w=\partial_{x_d}v$, with $v$ defined in Definition \ref{d1}, is lying in $H^{2s_1-1}(\Omega)$ and it satisfies 
\[\left\{\begin{aligned}(\mathcal A +p^\alpha)W(p)=-(\partial_{x_d}\mathcal A) V(p)+\partial_{x_d}\hat{F}(p),\quad &\textrm{in }\Omega,\\  W(p)=0,\quad &\textrm{on }\partial\Omega.\end{aligned}\right.\]
Note that here we use the fact that $W(p)\in H^{2s_1-1}(\Omega)$, with $2s_1-1>\frac{1}{2}$, for defining its trace on $\partial\Omega$. Moreover, applying Lemma \ref{l10} (see the Appendix) and the fact that $D(A^{1/2})=H^1_0(\Omega)$, we can extend $S(t)$, $t>0$, to a bounded operator from $H^{-1}(\Omega)$ to $L^2(\Omega)$ satisfying
$$\norm{S(t)}_{\mathcal B(H^{-1}(\Omega),L^2(\Omega))}\leq Ct^{\alpha/2-1}.$$
In addition, using the fact that $v\in L^2_{loc}(\R^+;D(A))\subset L^2_{loc}(\R^+;H^2(\Omega))$, we deduce that $(\partial_{x_d}\mathcal A) v\in L^2_{\textrm{loc}}(\R^+;L^2(\Omega))$. Thus, extending $F$ by zero to $(0,+\infty)\times\Omega$ and using the fact that  $$\inf\{\epsilon>0:\ e^{-\epsilon t}(\partial_{x_d}\mathcal A)v\in L^1(\R^+; H^{-1}(\Omega))\}=0,$$
we deduce that the function
$$\begin{aligned}w_1(t,\cdot)&:=\int_0^tS(t-s)[-(\partial_{x_d}\mathcal A)v(s,\cdot))+\partial_{x_d}F(s,\cdot)]ds\\
\ &=-\int_0^tS(t-s_1)(\partial_{x_d}\mathcal A)\left(\int_0^{\min(s_1,T)}\tilde{S}(s_1-s_2)f(s_2,\cdot)R(s_2,\cdot)ds_2\right)ds_1+\int_0^tS(t-s)f(s,\cdot)\partial_{x_d}R(s,\cdot)ds.\end{aligned}$$
is well defined and the Laplace transform in time of $w_1$ coincide with the one of $w$\footnote{ With the additional assumptions (H10),   (H11) and the fact that $\omega$ is $\mathcal C^2$ one can extend these arguments to problem \eqref{eq11}-\eqref{b2}, with $m=n=1$, by using the fact that, for any $s_1\in (0,1/2)$,  $\mathcal C^\infty_0(\Omega)$ is dense in $H^{s_1}(\Omega)$.}. This proves that
\begin{equation}\label{identity}\partial_{x_d}u=-\int_0^tS(t-s_1)(\partial_{x_d}\mathcal A)\left(\int_0^{s_1}\tilde{S}(s_1-s_2)f(s_2,\cdot)R(s_2,\cdot)ds_2\right)ds_1+\int_0^tS(t-s)f(s,\cdot)\partial_{x_d}R(s,\cdot)ds.\end{equation}
In view of \eqref{t3a},  for $w=\partial_{x_d}u$, we have 
\begin{equation}\label{identity4}-(\partial_{x_d}\mathcal A) u=\partial_{x_d}(\partial_{x_d}a_{dd}\partial_{x_d}u)+\sum_{j=1}^{d-1}[\partial_{x_j}(\partial_{x_d}a_{jd}\partial_{x_d}u)+\partial_{x_d}(\partial_{x_d}a_{jd}\partial_{x_j}u)]=B_1w+B_2u,\end{equation}
where 
\begin{equation}\label{bb1}B_1w:=\partial_{x_d}a_{dd}\partial_{x_d}w+2\sum_{j=1}^{d-1}\partial_{x_d}a_{jd}\partial_{x_j}w+\left(\sum_{j=1}^d\partial_{x_j}\partial_{x_d}a_{jd}\right)w,\end{equation}
\begin{equation}\label{bb2}B_2u=\sum_{j=1}^d(\partial_{x_d}^2a_{jd})\partial_{x_j}u.\end{equation}
Then, from \eqref{identity} and the above arguments, we deduce that $w$ solves the integral equation
\begin{equation}\label{int}\begin{aligned}w(t,\cdot)=&\int_0^tS(t-s)B_1w(s,\cdot)ds+\int_0^tS(t-s_1)B_2\left(\int_0^{s_1}\tilde{S}(s_1-s_2)f(s_2,\cdot)R(s_2,\cdot)ds_2\right)ds_1\\
\ &+\int_0^tS(t-s)f(s,\cdot)\partial_{x_d}R(s,\cdot)ds,\quad t\in(0,T).\end{aligned}\end{equation}

Now let us consider the following.

\begin{lem}\label{l6} The integral equation \eqref{int} admits a unique solution $w\in L^2(0,T;H^2(\Omega))$ satisfying
\begin{equation}\label{l6a}\norm{w(t,\cdot)}_{H^1(\Omega)}\leq C\int_0^t(t-s)^{\frac{\alpha}{2}-1}\norm{f(s,\cdot)}_{L^2(\omega)},\quad t\in(0,T),\end{equation}
with $C>0$ depending on $R$, $\mathcal A$, $T$. Moreover, we have $\partial_t^\alpha w\in L^2(Q)$.

\end{lem}
\begin{proof} We introduce the maps $\mathcal K_1,\ \mathcal M:L^2(0,T;H^1(\Omega))\longrightarrow L^2(0,T;H^1(\Omega))$, $\mathcal K_2:L^2((0,T)\times\omega))\longrightarrow L^2(0,T;H^1(\Omega))$ defined by
$$\mathcal K_1v(t,\cdot):=\int_0^tS(t-s)B_1v(s,\cdot)ds,$$
$$\mathcal K_2f(t,\cdot):=\int_0^tS(t-s_1)B_2\left(\int_0^{s_1}\tilde{S}(s_1-s_2)f(s_2,\cdot)R(s_2,\cdot)ds_2\right)ds_1+\int_0^tS(t-s)f(s,\cdot)\partial_{x_d}R(s,\cdot)ds$$
and $\mathcal Mv=\mathcal K_1v+\mathcal K_2f$. Applying Lemma \ref{l10} (see the Appendix) and the fact that both $D(A^{\frac{1}{2}})$ and $D(\tilde{A}^{\frac{1}{2}})$ embedded continuously into $H^1(\Omega)$,  we find
\begin{equation}\label{l6b}\norm{\mathcal K_1v(t,\cdot)}_{H^1(\Omega)}\leq C\norm{\mathcal K_1v(t,\cdot)}_{D(A^{\frac{1}{2}})}\leq C\int_0^t\frac{(t-s)^{\frac{\alpha}{2}-1}}{\Gamma(\alpha/2)}\norm{v(s,\cdot)}_{H^1(\Omega)}ds,\quad t\in(0,T)\end{equation}
\begin{equation}\label{l6c}\begin{aligned}\norm{\mathcal K_2f(t,\cdot)}_{H^1(\Omega)}&\leq C\left(\int_0^t(t-s)^{\frac{\alpha}{2}-1}\norm{f(s,\cdot)}_{L^2(\omega)}ds+\int_0^t\int_0^{s_1}(t-s_1)^{\frac{\alpha}{2}-1}(s_1-s_2)^{\frac{\alpha}{2}-1}\norm{f(s_2,\cdot)}_{L^2(\omega)}ds_2ds_1\right)\\
\ &\leq C\int_0^t(t-s)^{\frac{\alpha}{2}-1}\norm{f(s,\cdot)}_{L^2(\omega)}ds,\quad t\in(0,T).\end{aligned}\end{equation}
Following the proof of \cite[Proposition 1]{FK}, for $n\in\mathbb N$, we find by iteration
\begin{equation}\label{l6d}\norm{\mathcal K_1^nv(t,\cdot)}_{H^1(\Omega)}\leq  C^n\int_0^t\frac{(t-s)^{\frac{n\alpha}{2}-1}}{\Gamma(n\alpha/2)}\norm{v(s,\cdot)}_{H^1(\Omega)}ds,\quad t\in(0,T)\end{equation}
\begin{equation}\label{l6e}\norm{\mathcal K_1^{n-1}\mathcal K_2f(t,\cdot)}_{H^1(\Omega)}\leq  C^n\int_0^t\frac{(t-s)^{\frac{n\alpha}{2}-1}}{\Gamma(n\alpha/2)}\norm{f(s,\cdot)}_{L^2(\omega)}ds,\quad t\in(0,T).\end{equation}
It follows that for $n\in\mathbb N$ sufficiently large we have
$$\norm{\mathcal M^nv_1(t,\cdot)-\mathcal M^nv_2(t,\cdot)}_{H^1(\Omega)}\leq C^n\int_0^t\frac{(t-s)^{\frac{n\alpha}{2}-1}}{\Gamma(n\alpha/2)}\norm{v_1(s,\cdot)-v_2(s,\cdot)}_{H^1(\Omega)}ds,\quad t\in(0,T)$$
and an application of the Young inequality for convolution product implies that
$$\begin{aligned}\norm{\mathcal M^nv_1-\mathcal M^nv_2}_{L^2(0,T;H^1(\Omega))}&\leq C^n\left(\int_0^T\frac{t^{\frac{n\alpha}{2}-1}}{\Gamma(n\alpha/2)}dt\right)\norm{v_1-v_2}_{L^2(0,T;H^1(\Omega))}\\
\ &\leq \frac{C^nT^{\frac{n\alpha}{2}}}{\frac{n\alpha}{2}\Gamma(n\alpha/2)}\norm{v_1-v_2}_{L^2(0,T;H^1(\Omega))}\\
\ &\leq \frac{C^nT^{\frac{n\alpha}{2}+1}}{\Gamma(n\alpha/2+1)}\norm{v_1-v_2}_{L^2(0,T;H^1(\Omega))}.\end{aligned}$$
Then, using the fact that 
$$\lim_{n\to+\infty}\frac{C^nT^{\frac{n\alpha}{2}+1}}{\Gamma(n\alpha/2+1)}=0,$$
we deduce that there exists $n_0\in\mathbb N$ such that $\mathcal M^{n_0}$ is a contraction. Moreover, conditions \eqref{l6d}-\eqref{l6e} imply
$$\begin{aligned}\norm{\mathcal M^nv}_{L^2(0 ,T;H^1(\Omega))}&\leq \frac{C^nT^{\frac{n\alpha}{2}+1}}{\Gamma(n\alpha/2+1)}\norm{v}_{L^2(0,T;H^1(\Omega))}+\left(\sum_{k=1}^{n} \frac{C^kT^{\frac{k\alpha}{2}+1}}{\Gamma(k\alpha/2+1)}\right)\norm{f}_{L^2((0,T)\times\omega)}\\
\ &\leq \frac{C^nT^{\frac{n\alpha}{2}+1}}{\Gamma(n\alpha/2+1)}\norm{v}_{L^2(0,T;H^1(\Omega))}+\left(\sum_{k=1}^{\infty} \frac{C^kT^{\frac{k\alpha}{2}+1}}{\Gamma(k\alpha/2+1)}\right)\norm{f}_{L^2((0,T)\times\omega)}.\end{aligned}$$
Therefore, by eventually increasing the size of $n_0$, we deduce that $\mathcal M^{n_0}$ admits a unique fixed point $w\in L^2(0,T;H^1(\Omega))$ which by uniqueness of this fixed point is also a fixed point of $\mathcal M$. Moreover, in view of \eqref{l6d}-\eqref{l6e}, for $n_1\in\mathbb N$ satisfying $n_1\geq \frac{4}{\alpha}$ and for a.e $t\in(0,T)$, we have
$$\begin{aligned}\norm{w(t,\cdot)}_{H^1(\Omega))}&=\norm{\mathcal M^{n_1}w(t,\cdot)}_{H^1(\Omega))}\\
\ &\leq C\left(\int_0^t(t-s)^{\frac{n_1\alpha}{2}-1}\norm{w(s,\cdot)}_{H^1(\Omega)}ds+\int_0^t(t-s)^{\frac{\alpha}{2}-1}\norm{f(s,\cdot)}_{L^2(\omega)}ds\right)\\
\ &\leq C\left(T^{\frac{n_1\alpha}{2}-1}\int_0^t\norm{w(s,\cdot)}_{H^1(\Omega)}ds+\int_0^t(t-s)^{\frac{\alpha}{2}-1}\norm{f(s,\cdot)}_{L^2(\omega)}ds\right).\end{aligned}$$
Therefore, applying Gronwall inequality for function lying in $L^2(0,T)$ (see e.g \cite[Lemma 6.3]{BHKY}), we find
$$\norm{w(t,\cdot)}_{H^1(\Omega))}\leq C\left(\int_0^t(t-s)^{\frac{\alpha}{2}-1}\norm{f(s,\cdot)}_{L^2(\omega)}ds\right)e^{Ct},\quad t\in(0,T),$$
which clearly implies \eqref{l6a}. Finally, using the fact that
$$B_1w, s_1\mapsto B_2\left(\int_0^{s_1}\tilde{S}(s_1-s_2)f(s_2,\cdot)R(s_2,\cdot)ds_2\right), \partial_{x_d}F\in L^2(Q),$$
we deduce from Lemma \ref{l2} and assumption (H00), that $\mathcal M$ takes values in $L^2(0,T;H^2(\Omega))$ and therefore $w\in L^2(0,T;H^2(\Omega))$. In the same way, we prove that $\partial_t^\alpha w\in L^2(Q)$.\end{proof}

According to Lemma \ref{l6}, we have $u\in L^2(0,T;H^1(\Omega))$, $\mathcal A u,\ \partial_t^\alpha u\in L^2(Q)$ and $w=\partial_{x_d}u\in L^2(0,T;H^2(\Omega))$, $\partial_t^\alpha \partial_{x_d}u,\ \mathcal A\partial_{x_d}u\in L^2(Q)$. Combining this with \eqref{identity4}, we deduce that $$x_d\mapsto u(\cdot,\cdot,x_d)\in H^3((-\ell,\ell);L^2((0,T)\times\omega))\cap H^1(-\ell,\ell;L^2(0,T;H^2(\omega))),$$ 
$$x_d\mapsto\partial_t^\alpha u(\cdot,\cdot,x_d)\in H^1((-\ell,\ell);L^2((0,T)\times\omega)).$$  Then, \eqref{eq11} implies
\begin{equation}\label{identity2}\begin{aligned}a_{dd}\partial_{x_d}w(t,x',\ell)=a_{dd}\partial_{x_d}^2u(t,x',\ell)&=[\partial_t^\alpha u+(\mathcal A+a_{dd}\partial_{x_d}^2)u](t,x',\ell)-R(t,x',\ell)f(t,x')\\
\ &=R(t,x',\ell)[h(t,x')-f(t,x')],\quad (t,x')\in(0,T)\times\omega,\end{aligned}\end{equation}
with $h$ given by \eqref{t3b}. Fixing $r\in(3/4,1)$ and  $\tau_1:L^2(0,T;H^{2r}(\Omega))\ni y\longrightarrow \frac{a_{dd}\partial_{x_d}y(t,x',\ell)}{R(t,x',\ell)}\in L^2((0,T)\times\omega)$, we obtain 
$$f(t,\cdot)=h(t,\cdot)-\tau_1w(t,\cdot),\quad t\in(0,T).$$
Moreover, applying \eqref{identity4}, \eqref{int} and using the fact that $u,w\in L^2(0,T;H^1(\Omega))$, we get 
\begin{equation}\label{t3e}(\partial_{x_d}\mathcal A)\left(\int_0^{s_1}\tilde{S}(s_1-s_2)f(s_2,\cdot)R(s_2,\cdot)ds_2\right)ds_1=(\partial_{x_d}\mathcal A) u\in L^2(Q),\end{equation}
 which combined with (H00) implies 
\begin{equation}\label{identity3}\begin{aligned}f(t,\cdot)=&h(t,\cdot)+\tau_1\left(\int_0^tS(t-s_1)(\partial_{x_d}\mathcal A)\left(\int_0^{s_1}\tilde{S}(s_1-s_2)f(s_2,\cdot)R(s_2,\cdot)ds_2\right)ds_1\right)\\
&-\tau_1\left(\int_0^tS(t-s)f(s,\cdot)\partial_{x_d}R(s,\cdot)ds\right),\quad t\in(0,T).\end{aligned}\end{equation}
In view of (H00) by interpolation $D(A^r)$ embedded continuously into $H^{2r}(\Omega)$. Moreover, in view of Lemma \ref{l10} (see the Appendix), we have
$$\begin{aligned}&\norm{S(t-s_1)(\partial_{x_d}\mathcal A)\left(\int_0^{s_1}\tilde{S}(s_1-s_2)f(s_2,\cdot)R(s_2,\cdot)ds_2\right)}_{H^{2r}(\Omega)}\\
&\leq C\norm{S(t-s_1)(\partial_{x_d}\mathcal A)\left(\int_0^{s_1}\tilde{S}(s_1-s_2)f(s_2,\cdot)R(s_2,\cdot)ds_2\right)}_{D(A^r)}\\
&\leq C(t-s_1)^{\alpha(1-r)-1}\norm{(\partial_{x_d}\mathcal A)\left(\int_0^{s_1}\tilde{S}(s_1-s_2)f(s_2,\cdot)R(s_2,\cdot)ds_2\right)}_{L^2(\Omega)},\ s_1\in(0,t).\end{aligned}$$
On the other hand, applying Young inequality for convolution product, we obtain
$$\begin{aligned}&\int_0^T\int_0^t\norm{S(t-s_1)(\partial_{x_d}\mathcal A)\left(\int_0^{s_1}\tilde{S}(s_1-s_2)f(s_2,\cdot)R(s_2,\cdot)ds_2\right)}_{H^{2r}(\Omega)}ds_1dt\\
&\leq C\int_0^T\int_0^t(t-s_1)^{\alpha(1-r)-1}\norm{(\partial_{x_d}\mathcal A)\left(\int_0^{t}\tilde{S}(s_1-s_2)f(s_2,\cdot)R(s_2,\cdot)ds_2\right)}_{L^2(\Omega)}ds_1dt\\
&\leq CT^{1-r}\norm{(\partial_{x_d}\mathcal A)\left(\int_0^{t}\tilde{S}(s_1-s_2)f(s_2,\cdot)R(s_2,\cdot)ds_2\right)}_{L^1(0,T;L^2(\Omega))}\\
&\leq C\norm{(\partial_{x_d}\mathcal A)\left(\int_0^{t}\tilde{S}(s_1-s_2)f(s_2,\cdot)R(s_2,\cdot)ds_2\right)}_{L^2(Q)}<\infty.\end{aligned}$$
Thus, by Fubini theorem for a.e. $t\in(0,T)$, we have
$$s_1\longmapsto S(t-s_1)(\partial_{x_d}\mathcal A)\left(\int_0^{s_1}\tilde{S}(s_1-s_2)f(s_2,\cdot)R(s_2,\cdot)ds_2\right)\in L^1(0,t;H^{2r}(\Omega)).$$
In the same way, for a.e. $t\in(0,T)$, we get
$$s\longmapsto S(t-s)f(s,\cdot)\partial_{x_d}R(s,\cdot)\in L^1(0,t;H^{2r}(\Omega)).$$
Therefore, from \eqref{identity3},  for a.e. $t\in(0,T)$, we find 
\begin{equation}\label{identity5}\begin{aligned}f(t,\cdot)=&h(t,\cdot)+\int_0^t\tau_1S(t-s_1)(\partial_{x_d}\mathcal A)\left(\int_0^{s_1}\tilde{S}(s_1-s_2)f(s_2,\cdot)R(s_2,\cdot)ds_2\right)ds_1\\
&-\int_0^t\tau_1S(t-s)f(s,\cdot)\partial_{x_d}R(s,\cdot)ds.\end{aligned}\end{equation}
This proves \eqref{t3c}, let us prove that this problem is well-posed.  We fix the maps $\mathcal G$, $\mathcal H:L^2((0,T)\times\omega)\longrightarrow L^2((0,T)\times\omega)$, with
$$\begin{aligned}\mathcal H g(t,\cdot):=&\int_0^t\tau_1S(t-s_1)(\partial_{x_d}\mathcal A)\left(\int_0^{s_1}\tilde{S}(s_1-s_2)g(s_2,\cdot)R(s_2,\cdot)ds_2\right)ds_1\\
&-\int_0^t\tau_1S(t-s)g(s,\cdot)\partial_{x_d}R(s,\cdot)ds,\quad g\in L^2((0,T)\times\omega)\end{aligned}$$
and
$$\mathcal G g(t,\cdot):=h(t,\cdot)+\mathcal H g(t,\cdot),\quad g\in L^2((0,T)\times\omega).$$
Note that
\begin{equation}\label{t3h}\begin{aligned}\norm{\mathcal H g(t,\cdot)}_{L^2(\omega)}&\leq\int_0^t\norm{\tau_1S(t-s_1)(\partial_{x_d}\mathcal A)\left(\int_0^{s_1}\tilde{S}(s_1-s_2)g(s_2,\cdot)R(s_2,\cdot)ds_2\right)}_{L^2(\omega)}ds_1\\
&\ \ \ +\int_0^t\norm{\tau_1S(t-s)g(s,\cdot)\partial_{x_d}R(s,\cdot)}_{L^2(\omega)}ds\\
&\leq C\int_0^t\norm{S(t-s_1)(\partial_{x_d}\mathcal A)\left(\int_0^{s_1}\tilde{S}(s_1-s_2)g(s_2,\cdot)R(s_2,\cdot)ds_2\right)}_{H^{2r}(\Omega)}ds_1\\
&\ \ \ +C\int_0^t\norm{S(t-s)g(s,\cdot)\partial_{x_d}R(s,\cdot)}_{H^{2r}(\Omega)}ds\\
&\leq C(t^{\alpha(1-r)-1}\mathds{1}_{(0,+\infty)})*\norm{(\partial_{x_d}\mathcal A)\left(\mathds{1}_{(0,T)}(t)\int_0^{t}\tilde{S}(t-s)g(s,\cdot)R(s,\cdot)ds\right)}_{L^2(\Omega)}\\
&\ \ \ +C\norm{\partial_{x_d}R}_{L^\infty(Q)}(t^{\alpha(1-r)-1}\mathds{1}_{(0,+\infty)})*(\norm{g(\cdot,)}_{L^2(\omega)}\mathds{1}_{(0,T)}),\ t\in(0,T).\end{aligned}\end{equation}
On the other hand, in view of \eqref{identity4}, fixing $w_1\in L^2(0,T;H^1(\Omega))$ the solution of \eqref{int} with $f=g$ we obtain 
$$(-\partial_{x_d}\mathcal A)\left(\int_0^{t}\tilde{S}(t-s)g(s,\cdot)R(s,\cdot)ds\right)=B_1w_1(t,\cdot)+B_2\int_0^tS(t-s)gR(s,\cdot)ds,$$
where $B_1$ and $B_2$ are defined in formula \eqref{bb1}-\eqref{bb2}. Thus, applying Lemma \ref{l6} and Lemma \ref{l10} (see the Appendix), we obtain
$$\norm{(\partial_{x_d}\mathcal A)\left(\int_0^{t}\tilde{S}(t-s)g(s,\cdot)R(s,\cdot)ds\right)}_{L^2(\Omega)}\leq C\int_0^t(t-s)^{\frac{\alpha}{2}-1}\norm{g(s,\cdot)}_{L^2(\omega)}ds,\quad t\in(0,T).$$
Combining this with \eqref{t3h}, we get
$$\norm{\mathcal H g(t,\cdot)}_{L^2(\omega)}\leq C\int_0^t\frac{(t-s)^{\alpha(1-r)-1}}{\Gamma(\alpha(1-r))}\norm{g(s,\cdot)}_{L^2(\omega)}ds,\quad t\in(0,T).$$
By iteration, for all $n\in\mathbb N$, we deduce that 
$$\norm{\mathcal H^n g(t,\cdot)}_{L^2(\omega)}\leq C^n\int_0^t\frac{(t-s)^{n\alpha(1-r)-1}}{\Gamma(n\alpha(1-r))}\norm{g(s,\cdot)}_{L^2(\omega)}ds,\quad t\in(0,T)$$
and in a similar way to Lemma \ref{l6}, we deduce that there exists $n_2\in\mathbb N$ such that $\mathcal G^{n_2}$ is a contraction and  $\mathcal G$ admits a unique fixed point $f\in L^2((0,T)\times\omega)$ satisfying 
$$\norm{f(t,\cdot)}_{L^2(\omega)}=\norm{\mathcal Gf(t,\cdot)}_{L^2(\omega)}\leq \norm{h(t,\cdot)}_{L^2(\omega)}+C\int_0^t(t-s)^{\alpha(1-r)-1}\norm{f(s,\cdot)}_{L^2(\omega)}ds,\ t\in(0,T).$$
Therefore, in view of  Lemma \ref{l11} (see the Appendix), we have
$$\norm{f(t,\cdot)}_{L^2(\omega)}\leq C\left(\norm{h(t,\cdot)}_{L^2(\omega)}+\int_0^t(t-s)^{\alpha(1-r)-1}\norm{h(s,\cdot)}_{L^2(\omega)}ds\right),\ t\in(0,T),$$
and an application of the Young inequality yields
$$\norm{f}_{L^2((0,T)\times\omega)}\leq C\norm{h}_{L^2((0,T)\times\omega)}.$$
 This proves the well-posedness of \eqref{t3c} and the reconstruction of $f$ from the data $h$. Now let us consider the proof of the last part of Theorem \ref{t3}. For this purpose, we fix  $(h_1,f)\in L^2((0,T)\times\omega)\times L^2((0,T)\times\omega)$ satisfying \eqref{t3c} with $h=h_1$ and we consider $u\in L^2(0,T;H^1(\Omega))$ solving  \eqref{eq11}. Following the above argumentation, we can define $h\in L^2((0,T)\times\omega)$ given by \eqref{t3b} and $f$ solves \eqref{t3c}. This implies that
$h_1=f-\mathcal H f=h$. Therefore, we have $h=h_1$ and  the proof of Theorem \ref{t3} for  \eqref{eq11}-\eqref{b2}, with $m=1$ and $n=0$, is completed. Using similar arguments, one can check that this result is still true  for the problem \eqref{eq11}-\eqref{b2}, with $m=1$ and $n=1$.

\section{Proof of Theorem \ref{t1}}

This section is devoted to the proof of Theorem \ref{t1}. In contrast to the preceding section, for $u$ solving \eqref{eq12}-\eqref{b4},   $\partial_{x_d}u$ will not be a solution of an initial boundary value problem with homogeneous boundary condition. However, with suitable regularity conditions on $F$ we can consider the trace of $u$ at $\{x_d=\ell\}$. We will start by  proving Proposition \ref{l1}

\textbf{Proof of Proposition \ref{l1}.} We consider first the case $q=0$. Let $A$ be the operator $\mathcal A$ acting on $L^2(\tilde{\Omega})$ with domain 
$$D(A):=\{g\in H^1(\tilde{\Omega}):\ \mathcal Ag\in L^2(\tilde{\Omega}),\ u_{|\partial\omega\times(0,\ell)}=0,\ u_{|x_d=\ell}=0,\ \partial_{x_d}u_{|x_d=0}=0\}.$$
 The spectrum of $A$ consists of a non-decreasing sequence of strictly positive eigenvalues $(\lambda_n)_{n\geq1}$. Let us also introduce 
an orthonormal  basis in the Hilbert space $L^2(\tilde{\Omega})$ of eigenfunctions $(\phi_n)_{n\geq1}$ of $A$ associated with the non-decreasing sequence of  eigenvalues $(\lambda_n)_{n\geq1}$.  Then, for $n\in\mathbb N$, the solution of \eqref{eq12}-\eqref{b4} is given by
$$u_n(t):=\left\langle u(t,\cdot),\phi_n\right\rangle_{L^2(\tilde{\Omega})}=\int_0^t(t-s)^{\alpha-1}E_{\alpha,\alpha}(-\lambda_n(t-s)^\alpha)F_n(s)ds,$$
where $F_n(t):=\left\langle F(t,\cdot),\phi_n\right\rangle_{L^2(\tilde{\Omega})}$. Since $F_n\in W^{1,p}(0,T)$, $F_n(0)=0$, applying Lemma \ref{l12} (see the Appendix) and integrating by parts we find
$$u_n(t)=\lambda_n^{-1}F_n(t)-\lambda_n^{-1}\int_0^tE_{\alpha,1}(-\lambda_n(t-s)^\alpha)F_n'(s)ds.$$
Thus, we have $u=v+w$ where
$$v(t,\cdot)=A^{-1}F(t,\cdot),\quad w_n(t):=\left\langle w(t,\cdot),\phi_n\right\rangle_{L^2(\tilde{\Omega})}=-\lambda_n^{-1}\int_0^tE_{\alpha,1}(-\lambda_n(t-s)^\alpha)F_n'(s)ds.$$
Note that, for all $t\in[0,T]$, $v(t,\cdot)$ solves the boundary value problem
$$\left\{\begin{aligned}\mathcal A v(t,\cdot)=F(t,\cdot),\quad &\ \textrm{in }\tilde{\Omega},\\  
v(t,x)=0,\quad &x\in\partial\omega\times(0,\ell),\\
v(t,x',\ell)=0,\quad &x'\in\omega,\\
\partial_{x_d}v(t,x',0)=0,\quad &x'\in\omega.\end{aligned}\right.$$
Therefore, applying assumption $(\tilde{H})$ and the fact that $F\in\mathcal C([0,T];H^{2\gamma}(\tilde{\Omega}))$,  we deduce that $v\in \mathcal C([0,T];H^{2(1+\gamma)}(\tilde{\Omega}))$. Thus, in order to complete the proof, we only need to check that $w\in \mathcal C([0,T];H^{2(1+\gamma)}(\tilde{\Omega}))$. For this purpose, using the fact that $F_n'\in L^p(0,T)$ and $t\mapsto t^{\alpha-1}\in L^{p'}(0,T)$, with $1/p'=1-1/p$, one can check that $w_n\in\mathcal C([0,T])$, $n\in\mathbb N$. Moreover, fixing $m,n\in\mathbb N$, with $m<n$, and applying Lemma \ref{l10} (see the Appendix), we find
$$\begin{aligned}&\norm{\sum_{k=m}^n\lambda_k^{1+\gamma}w_k(t)\phi_k}_{L^2(\tilde{\Omega})}\\
&\leq \int_0^t\norm{\sum_{k=m}^n\lambda_k^{\gamma}E_{\alpha,1}(-\lambda_k(t-s)^\alpha)F_k'(s)\phi_k}_{L^2(\tilde{\Omega})}ds\\
&\leq C\int_0^t(t-s)^{-\alpha\gamma}\norm{\sum_{k=m}^nF_k'(s)\phi_k}_{L^2(\tilde{\Omega})}ds.\end{aligned}$$
On the other hand, since $\frac{1}{p}<1-\alpha\gamma$, we have $\frac{1}{p'}=1-\frac{1}{p}>\alpha\gamma$ and we deduce
$$\int_0^Ts^{-p'\alpha\gamma}ds<\infty.$$
Therefore, applying the H\"older inequality, we obtain
$$\norm{\sum_{k=m}^n\lambda_k^{1+\gamma}w_k(t)\phi_k}_{L^2(\tilde{\Omega})}\leq C\norm{\sum_{k=m}^nF_k'(s)\phi_k}_{L^p(0,T;L^2(\tilde{\Omega}))}.$$
Combining this with the fact that $F\in L^p(0,T;L^2(\tilde{\Omega}))$, we deduce that  
$$\lim_{m,n\to\infty}\norm{\sum_{k=m}^nw_k\phi_k}_{\mathcal C ([0,T];D(A^{1+\gamma}))}=0.$$
Thus, the sequence
$$\sum_{n\in\mathbb N}w_n\phi_n$$
is a Cauchy sequence and therefore a convergent sequence of $\mathcal C ([0,T];D(A^{1+\gamma}))$. Since this sequence converges to $w$ in the sense of $\mathcal C([0,T];L^2(\tilde{\Omega}))$, we deduce that $w\in\mathcal C ([0,T];D(A^{1+\gamma}))$ and, in view of $(\tilde{H})$, we deduce that $w\in \mathcal C([0,T];H^{2(1+\gamma)}(\tilde{\Omega}))$. Now let us prove that $u\in \mathcal C^1([0,T];H^{2\gamma}(\tilde{\Omega}))$.
Note first that $u_n\in\mathcal C^1([0,T])$ with
$$u_n'(t)=F_n(0)t^{\alpha-1}E_{\alpha,\alpha}(-\lambda_nt^\alpha)+\int_0^ts^{\alpha-1}E_{\alpha,\alpha}(-\lambda_ns^\alpha)F_n'(t-s)ds=\int_0^ts^{\alpha-1}E_{\alpha,\alpha}(-\lambda_ns^\alpha)F_n'(t-s)ds.$$
Here we have used the fact that $F_{|t=0}=0$. Repeating the above arguments and using the fact that $F\in W^{1,p}(0,T;L^2(\tilde{\Omega}))$ with $\frac{1}{p}<\alpha(1-\gamma)$, we deduce that $\partial_tu\in\mathcal C ([0,T];D(A^{\gamma}))\subset \mathcal C([0,T];H^{2\gamma}(\tilde{\Omega}))$. Therefore, we have $u\in \mathcal C^1([0,T];H^{2\gamma}(\tilde{\Omega}))$. 

Now let us consider the case $q\neq0$. We introduce the map
$$\mathcal G(v):=\int_0^tJ(t-s)F(s)ds-\int_0^tJ(t-s)q(s)v(s)ds$$
defined on $\mathcal C([0,T];H^{2(1+\gamma)}( \tilde{\Omega}))\cap \mathcal C^1([0,T];H^{2\gamma}(\tilde{\Omega}))$ with $J(t)$ given by
$$J(t)h:=\sum_{n=1}^\infty t^{\alpha-1}E_{\alpha,\alpha}(-\lambda_n t^\alpha)\left\langle h,\phi_n\right\rangle\phi_n,\quad h\in L^2(\tilde{\Omega}),\ t>0.$$
Then, using a classical fixed point argument combined with the preceding analysis, we deduce that $\mathcal G$ admits a unique fixed point lying in $\mathcal C([0,T];H^{2(1+\gamma)}( \tilde{\Omega}))\cap \mathcal C^1([0,T];H^{2\gamma}(\tilde{\Omega}))$ which will be the solution of \eqref{eq12}-\eqref{b4}. This completes the proof of the lemma.\qed

From now on and in all the remaining part of this section, we  assume that $\alpha\in(0,1]$ and that the conditions of Proposition \ref{l1} are fulfilled. We consider $\tilde{u}$ defined on $(0,T)\times\omega\times(-\ell,\ell)$ by
$$\tilde{u}(t,x',-x_d):=\tilde{u}(t,x',x_d)=u(t,x',x_d),\quad (t,x',x_d)\in(0,T)\times\omega\times(0,\ell).$$
Then, it is clear that $\tilde{u}_{|(0,T)\times \tilde{\Omega}}\in \mathcal C([0,T];H^{2(1+\gamma)}( \tilde{\Omega}))\cap \mathcal C^1([0,T];H^\gamma(\tilde{\Omega}))$ and
$\tilde{u}_{|(0,T)\times \omega\times(-\ell,0)}\in \mathcal C([0,T];H^{2(1+\gamma)}( \omega\times(-\ell,0)))\cap \mathcal C^1([0,T];H^\gamma(\omega\times(-\ell,0)))$.
Moreover, we have
$$\partial_{x_d}^{2k}\tilde{u}(t,x',0_+)=\partial_{x_d}^{2k}\tilde{u}(t,x',0_-),\quad \partial_{x_d}\tilde{u}(t,x',0_+)=0=\partial_{x_d}\tilde{u}(t,x',0_-),\quad (t,x')\in(0,T)\times\omega,\ k=0,1.$$
Therefore,  we deduce that $\tilde{u}\in \mathcal C([0,T];H^{3}(\Omega))\cap \mathcal C^1([0,T];H^{2\gamma}(\Omega))$ and, thanks to \eqref{co2}, $\tilde{u}$ solves
$$
\left\{\begin{aligned}\partial_t^\alpha \tilde{u}+\mathcal A \tilde{u}+q(t,x')\tilde{u}=f(t,x')\tilde{R}(t,x),\quad &(t,x',x_d)\in 
(0,T)\times \omega\times(-\ell,\ell),\\  
\tilde{u}(t,x)=0,\quad &(t,x)\in(0,T)\times \partial\omega\times(-\ell,\ell),\\ 
\tilde{u}(t,x',\pm\ell)=0\quad &(t,x')\in(0,T)\times \omega,\\
\partial_t^k\tilde{u}(0,x)=0,\quad &x\in\omega\times(-\ell,\ell),\ k=0,\ldots,m_\alpha,\end{aligned}\right.
$$
where $\tilde{R}$ is the extension of $R$ to $(0,T)\times\omega\times(-\ell,\ell)$ given by
$$\tilde{R}(t,x',x_d):=R(t,x',-x_d),\quad (t,x',x_d)\in(0,T)\times\omega\times(-\ell,0).$$
Fixing $v=\partial_{x_d}\tilde{u}$  and using the fact that $\tilde{u}\in \mathcal C([0,T];H^{3}(\Omega))\cap \mathcal C^1([0,T];H^{2\gamma}(\Omega))$, we deduce that $v$ solves the problem
\begin{equation}\label{eq2}
\left\{\begin{aligned}\partial_t^\alpha v+\mathcal A v+q(t,x')v=H(t,x',x_d)+f(t,x')\partial_{x_d}\tilde{R}(t,x),\quad &(t,x',x_d)\in 
(0,T)\times \omega\times(-\ell,\ell),\\  
v(t,x)=0,\quad &(t,x)\in(0,T)\times \partial\omega\times(-\ell,\ell),\\ 
v(t,x',\pm\ell)=\partial_{x_d}\tilde{u}(t,x',\pm\ell)=\pm\partial_{x_d}u(t,x',\ell),\quad &(t,x')\in(0,T)\times \omega,\\
\partial_t^kv(0,x)=0,\quad &x\in\omega\times(-\ell,\ell),\ k=0,\ldots,m_\alpha,\end{aligned}\right.
\end{equation}
with
$$H(t,x',x_d):= \sum_{i,j=1}^d\partial_{x_i}\left(\partial_{x_d}a_{ij}(x)\partial_{x_j}\tilde{u}\right).$$
Note that here since $R,\partial_{x_d}R\in L^\infty((0,T)\times\tilde{\Omega})$, we have $\tilde{R},\partial_{x_d}\tilde{R}\in L^\infty((0,T)\times\Omega)$.
We are now in position to prove Theorem \ref{t1}.

\textbf{Proof of Theorem \ref{t1}.} In all this proof $C$ denotes a generic constant depending on $\alpha$, $c$, $T$, $\Omega$,  $\mathcal A$, $q$,  $d$.
According to Lemma \ref{l1} the solution $u$ of \eqref{eq12}-\eqref{b4} is lying in $\mathcal C([0,T];H^{2(1+\gamma)}( \tilde{\Omega}))\cap \mathcal C^1([0,T];H^\gamma(\tilde{\Omega}))$ and $v=\partial_{x_d}\tilde{u}$, with $\tilde{u}$ the even extension of $u$ to $(0,T)\times\omega\times(-\ell,\ell)$, which solves \eqref{eq2}.
Note first that projecting the equation \eqref{eq12} in $(0,T)\times\omega\times\{\ell\}$ and using the fact that, for all $(t,x')\in(0,T)\times\omega$, we have $u(t,x',\ell)=0$, we deduce that
$$f(t,x')R(t,x',\ell)=[a_{d,d}\partial_{x_d}^2u+2\sum_{j=1}^{d-1}a_{jd}\partial_{x_j}\partial_{x_d}u+\sum_{j=1}^{d}\partial_{x_j}a_{jd}\partial_{x_d}u](t,x',\ell),\quad (t,x')\in(0,T)\times\omega.$$
Combining this with \eqref{t1a} and the fact that $v=\partial_{x_d}u$, for all $t\in(0,T)$, we deduce that 
 \begin{equation}\label{t1c}\begin{aligned}\norm{f(t,\cdot)}_{L^2(\omega)}&\leq c^{-1}C(\norm{\partial_{x_d}v(t,\cdot,\ell)}_{L^2(\omega)}+\norm{\partial_{x_d}u(t,\cdot,\ell)}_{H^1(\omega)})\\
\ &\leq C(\norm{\partial_{x_d}v(t,\cdot,\ell)}_{L^2(\omega)}+\norm{\partial_{x_d}u(t,\cdot,\ell)}_{H^{\frac{3}{2}}(\omega)}).\end{aligned}\end{equation}
Thus, the proof of \eqref{t1b} will be completed if we can derive a suitable estimate of $\norm{\partial_{x_d}v(t,\cdot,\ell)}_{L^2(\omega)}$.
For this purpose, we decompose $v$ into $v=v_1+v_2$, where $v_1$ solves
\begin{equation}\label{eq3}
\left\{\begin{aligned}\partial_t^\alpha v_1+\mathcal A v_1=0,\quad &\textrm{ in } 
(0,T)\times \omega\times(-\ell,\ell),\\  
v_1(t,x)=0,\quad &(t,x)\in(0,T)\times \partial\omega\times(-\ell,\ell),\\ 
v_1(t,x',\pm\ell)=\partial_{x_d}\tilde{u}(t,x',\pm\ell)=\pm\partial_{x_d}u(t,x',\ell),\quad &(t,x')\in(0,T)\times \omega,\\
v_1(0,x)=0,\quad &x\in\omega\times(-\ell,\ell),\end{aligned}\right.
\end{equation}
and $v_2$ solves
\begin{equation}\label{eq4}
\left\{\begin{aligned}\partial_t^\alpha v_2+\mathcal A v_2=-qv(t,x)+H(t,x',x_d)+f(t,x')\partial_{x_d}\tilde{R}(t,x),\quad &(t,x',x_d)\in 
(0,T)\times \omega\times(-\ell,\ell),\\  
v_2(t,x)=0,\quad &(t,x)\in(0,T)\times \partial\omega\times(-\ell,\ell),\\ 
v_2(t,x',\pm\ell)=\partial_{x_d}\tilde{u}(t,x',\pm\ell)=0,\quad &(t,x')\in(0,T)\times \omega,\\
v_2(0,x)=0,\quad &x\in\omega\times(-\ell,\ell).\end{aligned}\right.
\end{equation}
Since $u\in \mathcal C^1([0,T];H^{2\gamma}(\tilde{\Omega}))\cap \mathcal C([0,T];H^{2(1+\gamma)}(\tilde{\Omega}))$, we have $$(t,x')\mapsto \partial_{x_d}u(t,x',\ell)\in \mathcal C^1([0,T];H^{2\gamma-\frac{3}{2}}(\omega))\cap \mathcal C([0,T];H^{\frac{1}{2}+2\gamma}(\omega))$$ and $\partial_{x_d}u(0,\cdot,\ell)=0$. therefore, using a classical  lifting argument (e.g. \cite[Theorem 8.3, Chapter 1]{LM1}), we can find $G\in \mathcal C^1([0,T];H^{2\gamma-1}(\Omega))\cap \mathcal C([0,T];H^{2\gamma+1}(\Omega))$ such that 
$$G(t,x',\ell)=\partial_{x_d}u(t,x',\ell),\quad (t,x')\in(0,T)\times\omega,$$
$$G(0,x',x_d)=0,\quad (x',x_d)\in\omega\times(-\ell,\ell),$$
 \begin{equation}\label{t1d}\begin{aligned}&\norm{G}_{W^{1,\infty}(0,T;L^2(\omega\times(-\ell,\ell)))}+\norm{G}_{L^\infty(0,T;H^2(\omega\times(-\ell,\ell)))}\\
&\leq C(\norm{\partial_{x_d}u(\cdot,\cdot,\ell)}_{L^\infty(0,T;H^{\frac{3}{2}}(\omega))}+\norm{\partial_{x_d}u(\cdot,\cdot,\ell)}_{W^{1,\infty}(0,T;H^\delta(\omega))}).\end{aligned}\end{equation}
Therefore, we can decompose $v_1$ into $v_1=G+w_1$ with $w_1$ solving
$$\left\{\begin{aligned}\partial_t^\alpha w_1+\mathcal A w_1=G_1,\quad &\textrm{ in } 
(0,T)\times \omega\times(-\ell,\ell),\\  
w_1(t,x)=0,\quad &(t,x)\in(0,T)\times \partial\omega\times(-\ell,\ell),\\ 
w_1(t,x',\pm\ell)=0,\quad &(t,x')\in(0,T)\times \omega,\\
w_1(0,x)=0,\quad &x\in\omega\times(-\ell,\ell),\end{aligned}\right.$$
where $G_1=-\partial_t^\alpha G-\mathcal AG\in \mathcal C([0,T];L^2(\Omega))$. Thus, we have 
$$w_1(t,\cdot)=\int_0^tS(s)G_1(t-s)ds,$$
where  $S(t)$ corresponds to the operator valued function defined in the proof of Theorem \ref{t3}.
Therefore, applying (H00) and  Lemma \ref{l10} (see the Appendix), we deduce that  $w_1\in\mathcal C([0,T];H^{2\gamma}(\Omega))$ with
$$\norm{w_1}_{L^\infty(0,T;H^{2\gamma}(\Omega))}\leq C\left(\int_0^Ts^{\alpha(1-\gamma)-1}ds\right)\norm{G_1}_{L^\infty(0,T;L^2(\Omega))}.$$
It follows that
$$\begin{aligned}\norm{v_1}_{L^\infty(0,T;H^{2\gamma}(\Omega))}&\leq C(\norm{G_1}_{L^\infty(0,T;L^2( \Omega))}+\norm{G}_{L^\infty(0,T;H^2( \Omega))})\\
\ &\leq C(\norm{G}_{W^{1,\infty}(0,T;L^2(\omega\times(-\ell,\ell)))}+\norm{G}_{L^\infty(0,T;H^2(\omega\times(-\ell,\ell)))})\end{aligned}$$
and combining this with  \eqref{t1d}, we get 
\begin{equation}\label{tt}\norm{v_1}_{L^\infty(0,T;H^{2\gamma}(\omega\times(-\ell,\ell))}\leq C(\norm{\partial_{x_d}u(\cdot,\cdot,\ell)}_{L^\infty(0,T;H^{\frac{3}{2}}(\omega))}+\norm{\partial_{x_d}u(\cdot,\cdot,\ell)}_{W^{1,\infty}(0,T;H^\delta(\omega))}).\end{equation}
Moreover, the estimate
$$\norm{\partial_{x_d}v_1(\cdot,\cdot,\ell)}_{L^\infty(0,T;L^2(\omega))}\leq C\norm{v_1}_{L^\infty(0,T;H^{2\gamma}(\omega\times(-\ell,\ell))},$$
implies
 \begin{equation}\label{t1e}\norm{\partial_{x_d}v_1(\cdot,\cdot,\ell)}_{L^\infty(0,T;L^2(\omega))}\leq C(\norm{\partial_{x_d}u(\cdot,\cdot,\ell)}_{L^\infty(0,T;H^{\frac{3}{2}}(\omega))}+\norm{\partial_{x_d}u(\cdot,\cdot,\ell)}_{W^{1,\infty}(0,T;H^\delta(\omega))}).\end{equation}
On the other hand, we have 
$$v_2(t,\cdot)=\int_0^tS(t-s)[-qv(s,\cdot)+H(s,\cdot)+\partial_{x_d}\tilde{F}(s,\cdot)]ds.$$
 Thus, applying (H00) and repeating the arguments of Lemma \ref{l1}, we get
$$\norm{v_2(t,\cdot)}_{H^{2\gamma}(\omega\times(-\ell,\ell))}\leq C\int_0^t(t-s)^{\alpha(1-\gamma)-1}[\norm{f(s,\cdot)}_{L^2(\omega)}+\norm{v(s,\cdot)}_{L^2(\omega\times(-\ell,\ell))}+\norm{H(s,\cdot)}_{L^2(\omega\times(-\ell,\ell))}]ds$$
which, for all $t\in(0,T]$, implies that
$$\norm{\partial_{x_d}v_2(t,\cdot,\ell)}_{L^2(\omega)}\leq C\int_0^t(t-s)^{\alpha(1-\gamma)-1}[\norm{f(s,\cdot)}_{L^2(\omega)}+\norm{v(s,\cdot)}_{L^2(\omega\times(-\ell,\ell))}+\norm{H(s,\cdot)}_{L^2(\omega\times(-\ell,\ell))}]ds.$$
In light  of \eqref{t3a}, we get
$$\begin{aligned}H&=\partial_{x_d}a_{dd}\partial_{x_d}^2\tilde{u}+2\sum_{j=1}^{d-1}\partial_{x_d}a_{jd}\partial_{x_j}\partial_{x_d}\tilde{u}+\sum_{j=1}^{d}\partial_{x_j}\partial_{x_d}a_{jd}\partial_{x_d}\tilde{u}+\sum_{j=1}^{d-1}\partial_{x_d}^2a_{jd}\partial_{x_j}\tilde{u}\\
\ &=\partial_{x_d}a_{dd}\partial_{x_d}v+2\sum_{j=1}^{d-1}\partial_{x_d}a_{jd}\partial_{x_j}v+\left(\sum_{j=1}^{d}\partial_{x_j}\partial_{x_d}a_{jd}\right)v+\sum_{j=1}^{d-1}\partial_{x_d}^2a_{jd}\partial_{x_j}\tilde{u}\end{aligned}$$
and it follows that
$$\norm{H(t,\cdot)}_{{L^2(\Omega)}}\leq C(\norm{v(t,\cdot)}_{H^1(\Omega)}+\norm{u(t,\cdot)}_{H^1(\Omega)}),\quad t\in(0,T].$$
In view of the equation satisfied by $v$ and according to the above arguments as well as the arguments used in Proposition \ref{l1}, for all $t\in(0,T]$, we get
$$\begin{aligned}&\norm{v(t,\cdot)}_{H^1(\Omega)}+\norm{u(t,\cdot)}_{H^1(\Omega)}\\
&\leq C(\norm{f}_{L^\infty(0,t;L^2(\omega))}+\norm{\partial_{x_d}u(\cdot,\cdot,\ell)}_{L^\infty(0,T;H^{\frac{3}{2}}(\omega))}+\norm{\partial_{x_d}u(\cdot,\cdot,\ell)}_{W^{1,\infty}(0,T;H^\delta(\omega))}).\end{aligned}$$
Thus, we find 
$$\norm{H(t,\cdot)}_{L^2(\Omega)}\leq C(\norm{f}_{L^\infty(0,t;L^2(\omega))}+\norm{\partial_{x_d}u(\cdot,\cdot,\ell)}_{L^\infty(0,T;H^{\frac{3}{2}}(\omega))}+\norm{\partial_{x_d}u(\cdot,\cdot,\ell)}_{W^{1,\infty}(0,T;H^\delta(\omega))})$$
and it follows that
$$\begin{aligned}&\norm{\partial_{x_d}v_2(t,\cdot,\ell)}_{L^2(\omega)}\\
&\leq C\left(\int_0^t(t-s)^{\alpha(1-\gamma)-1}\norm{f}_{L^\infty(0,s;L^2(\omega))}ds+\norm{\partial_{x_d}u(\cdot,\cdot,\ell)}_{L^\infty(0,T;H^{\frac{3}{2}}(\omega))}+\norm{\partial_{x_d}u(\cdot,\cdot,\ell)}_{W^{1,\infty}(0,T;H^\delta(\omega))}\right)\end{aligned}$$
Combining this with \eqref{t1c} and \eqref{t1e}, we find
$$\begin{aligned}\norm{f(t,\cdot)}_{L^2(\omega)}\leq& C(\norm{\partial_{x_d}u(\cdot,\cdot,\ell)}_{L^\infty(0,T;H^{\frac{3}{2}}(\omega))}+\norm{\partial_{x_d}u(\cdot,\cdot,\ell)}_{W^{1,\infty}(0,T;H^\delta(\omega))})\\
\ &+C\int_0^t(t-s)^{\alpha(1-\gamma)-1}\norm{f}_{L^\infty(0,s;L^2(\omega))}ds,\quad\quad\quad t\in(0,T),\end{aligned}$$
which clearly implies
$$\begin{aligned}\norm{f}_{L^\infty(0,t;L^2(\omega))}\leq& C(\norm{\partial_{x_d}u(\cdot,\cdot,\ell)}_{L^\infty(0,T;H^{\frac{3}{2}}(\omega))}+\norm{\partial_{x_d}u(\cdot,\cdot,\ell)}_{W^{1,\infty}(0,T;H^\delta(\omega))})\\
\ &+C\int_0^t(t-s)^{\alpha(1-\gamma)-1}\norm{f}_{L^\infty(0,s;L^2(\omega))}ds,\quad t\in(0,T).\end{aligned}$$
Then,  Lemma \ref{l11} (see the Appendix) implies that
$$\norm{f}_{L^\infty(0,t;L^2(\omega))}\leq C(\norm{\partial_{x_d}u(\cdot,\cdot,\ell)}_{L^\infty(0,T;H^{\frac{3}{2}}(\omega))}+\norm{\partial_{x_d}u(\cdot,\cdot,\ell)}_{W^{1,\infty}(0,T;H^\delta(\omega))}),\quad t\in(0,T)$$
from which we deduce \eqref{t1b}.\qed

\subsection{Application to the recovery of coefficients}

Consider $v$ the solution of the problem

$$
\left\{\begin{aligned}\partial_t^\alpha v+\mathcal A v+q(t,x')v=0,\quad &(t,x)\in 
(0,T)\times \tilde{\Omega},\\  
v(t,x', \ell)=h_0(t,x')\quad &(t,x')\in(0,T)\times \omega,\\
v(t,x)=h_1 ,\quad &(t,x)\in(0,T)\times \partial\omega\times(-\ell,\ell),\\ 
\partial_{x_d}v(t,x', 0)=0\quad &(t,x')\in(0,T)\times \omega,\\
v(0,x)=w_0(x),\quad &x\in\tilde{\Omega},\end{aligned}\right.$$
with $w_0$, $h_k$,  $k=0,1$, given by \eqref{trace}. Then, we can write $v=H+y$ where $y$ solves
$$\left\{\begin{aligned}\partial_t^\alpha y+\mathcal A y+q(t,x')y=H_1(t,x),\quad &(t,x)\in 
(0,T)\times\tilde{\Omega},\\  
y(t,x)=0 ,\quad &(t,x)\in(0,T)\times \partial\omega\times(0,\ell),\\ 
y(t,x', \ell)=0\quad &(t,x')\in(0,T)\times \omega,\\
\partial_{x_d}y(t,x', 0)=0\quad &(t,x')\in(0,T)\times \omega,\\
y(0,x)=0,\quad &x\in\tilde{\Omega},\end{aligned}\right.$$
where $H_1=-(\partial_t^\alpha+\mathcal A +q)H\in \mathcal C([0,T];H^{2\gamma}(\tilde{\Omega}))\cap W^{{1},{p}}(0,T; L^2(\tilde{\Omega}))$. Therefore, using the fact that, thanks to \eqref{trace}, we have 
$$H_1(0,x)=0,\quad x\in\tilde{\Omega}$$
and  repeating the arguments of the preceding section, we can prove that $v\in \mathcal C([0,T];H^{2(1+\gamma)}(\tilde{\Omega}))$ and for $d\leq3$, the Sobolev embedding theorem implies that $v,\partial_{x_d}v\in\mathcal C\left([0,T]\times\overline{\tilde{\Omega}}\right)$. Applying the previous results about recovery of source terms we can complete the proof of Corollary \ref{c1}.

\textbf{Proof of Corollary \ref{c1}.}
Let $u=v_1-v_2$ and notice that $u$ solves \eqref{eq11} with $q=q_1$, $f=(q_2-q_1)$ and $R=v_2$. Then, using the fact that  $v_2,\partial_{x_d}v_2\in\mathcal C\left([0,T]\times\overline{\tilde{\Omega}}\right)$ and the fact that, thanks to \eqref{t1a}, \eqref{t5a}, we are in position to apply Theorem \ref{t1} from which we deduce \eqref{t5a}.\qed

\section*{Appendix}

In this appendix we recall several classical result about fractional diffusion equation and properties of Mittag-Leffler function.

We start by recalling a property of Mittag-Leffler function which follows from formula (1.148) of \cite[Theorem 1.6]{P}, one can check the following properties of the Mittag-Leffler function.
\begin{lem}\label{l10} Let $\lambda>0$, $\alpha\in(0,2)$ and $\beta>0$. Then, we have
$$|E_{\alpha,\beta}(-\lambda t^\alpha)|\leq \frac{C}{1+\lambda t^\alpha},\quad t>0,$$
with $C>0$ independent of $t$ and $\lambda$.
\end{lem}

Now let us consider the following  result which can be deduced  from other  known results (see e.g \cite[Theorem 2.2]{SY} and \cite[Theorem 1.3]{LKS}) considered for $\alpha\in(0,1)$ that we extend to $\alpha\in(0,2)$.
\begin{lem}\label{l2} Let $F\in L^2(Q)$, $\alpha\in(0,2)$ and $m,n=0,1$. Then problem \eqref{eq11}-\eqref{b2} admits a unique solution $u\in L^2(0,T;H^1(\Omega))$, satisfying $\mathcal A u$, $\partial_t^\alpha u\in L^2(Q)$ and the following estimate holds true 
\begin{equation}\label{l2a}\norm{\partial_t^\alpha u}_{L^2(Q)}+\norm{\mathcal Au}_{L^2(Q)}+\norm{u}_{L^2(Q)}\leq C\norm{F}_{L^2(Q)}.\end{equation}
\end{lem}
\begin{proof}  We prove this result  for sake of completeness.
 We fix $A$ the operator $\mathcal A$ acting on $L^2(\Omega)$ with the boundary condition \eqref{b1}-\eqref{b2}. We fix also  the non-decreasing sequence of non-negative eigenvalues $(\lambda_k)_{k\geq1}$ and associated eigenfunctions $(\phi_k)_{k\geq1}$ of $A$. Then, we consider
$$k_0:=\min\{k\in\mathbb N:\ \lambda_k>0\}.$$
Since $u$ solves \eqref{eq11}-\eqref{b2}, we have
$$u_k(t):=\left\langle u(t,\cdot),\phi_k\right\rangle_{L^2(\Omega)}=\int_0^t(t-s)^{\alpha-1}E_{\alpha,\alpha}(-\lambda_k(t-s)^\alpha)F_k(s)ds,$$
where $F_k(t):=\left\langle F(t,\cdot),\phi_k\right\rangle_{L^2(\Omega)}$.
Therefore, we have 
$$u_k(t)=(G_k\mathds{1}_{(0,T)})*(F_k\mathds{1}_{(0,T)})$$
with  $G_k(t)=t^{\alpha-1}E_{\alpha,\alpha}(-\lambda_kt^\alpha)$. Then, we find 
$$\int_0^T \sum_{k=1}^\infty \lambda_k^2|u_k(t)|^2dt=\sum_{k=1}^\infty \lambda_k^2\norm{(G_k\mathds{1}_{(0,T)})*(F_k\mathds{1}_{(0,T)})}_{L^2(0,T)}^2$$
and an application of the Young inequality yields
$$\int_0^T \sum_{k=1}^\infty\lambda_k^2|u_k(t)|^2dt=\sum_{k=1}^\infty\lambda_k^2\norm{G_k}_{L^1(0,T)}^2\norm{F_k}_{L^2(0,T)}^2.$$
On the other hand, we have 
$$\norm{G_k}_{L^1(0,T)}=\int_0^Ts^{\alpha-1}E_{\alpha,\alpha}(-\lambda_k s^{\alpha})ds=\lambda_k^{-1}\int_0^{\lambda_k^{\frac{1}{\alpha}}T}s^{\alpha-1}E_{\alpha,\alpha}(- s^{\alpha})ds,\quad k\geq k_0.$$
Combining this with Lemma \ref{l10} (see the Appendix), we get
$$\norm{G_k}_{L^1(0,T)}\leq C\lambda_k^{-1}\int_0^{+\infty}\frac{s^{\alpha-1}}{1+s^{2\alpha}}ds\leq C\lambda_k^{-1},\quad k\geq k_0+1.$$
It follows that
$$\begin{aligned}\int_0^T \sum_{k=1}^\infty (1+\lambda_k^2)|u_k(t)|^2dt&= C\int_0^T \sum_{k=1}^{k_0} |u_k(t)|^2dt+\int_0^T \sum_{k=k_0+1}^\infty (1+\lambda_k^2)|u_k(t)|^2dt\\
\ &\leq C\sum_{k=1}^\infty\norm{F_k}_{L^2(0,T)}^2\leq C\norm{F}_{L^2(Q)}^2\end{aligned}$$
and we get
$$\norm{u}_{L^2(0,T;D(A))}\leq C\norm{F}_{L^2(Q)}.$$
 In the same way, we find
$$\partial_t^\alpha u=-\mathcal A u+F\in L^2(Q)$$
which implies at the same time that $\partial_t^\alpha u\in L^2(Q)$ and \eqref{l2a}.
This proves \eqref{l2a}.\end{proof}

Let us also consider the following Gronwall type of inequality which can be find in \cite[Lemma 3]{FK} (see also \cite[Theorem 1]{YGD}).

\begin{lem}\label{l11}
  Let $C,r>0$ and $h,d\in L^1(0,T)$ be nonnegative functions satisfying
\[
	h(t)\le d(t)+C\int_0^t(t-s)^{r-1}h(s)ds,\quad t\in(0,T).
\]
  Then there exists $C_1>0$ such that 
\[
	h(t)\le d(t)+C_1\int_0^t(t-s)^{r-1}d(s)ds,\quad t\in(0,T).
\]

\end{lem}

Finally let us recall, a result borrowed from \cite[Lemma 3.2]{SY}.

\begin{lem}\label{l12}
Let $\lambda>0$, $\alpha>0$ and  $m\geq1 $ be a positive integer. Then we have
$$
\frac{d^m}{dt^m}E_{\alpha,1}(-\lambda t^{\alpha})
= -\lambda t^{\alpha-m}E_{\alpha,\alpha-m+1}(-\lambda t^{\alpha}), \quad
t > 0
$$
and
$$
\frac{d}{dt}(tE_{\alpha,2}(-\lambda t^{\alpha}))
= E_{\alpha,1}(-\lambda t^{\alpha}), \quad t > 0.
$$
\end{lem}

\section*{ Acknowledgments}
The work of the first author is partially supported by  the French National
Research Agency ANR (project MultiOnde) grant ANR-17-CE40-0029.
The second author is supported by Grant-in-Aid for Scientific Research (S) 
15H05740 of Japan Society for the Promotion of Science and
by The National Natural Science Foundation of China 
(no. 11771270, 91730303), and the "RUDN University Program 5-100".

\end{document}